\documentclass[12pt]{amsart}
\usepackage{amsmath}
\usepackage{amsrefs}
\usepackage{times}
\usepackage{graphicx}
\usepackage{amssymb}
\usepackage{amsthm}
\usepackage{textcomp}
\usepackage[T1]{fontenc}
\usepackage[margin=1.0in]{geometry}

\begin{document}
\newtheorem{theorem}{Theorem}
\newtheorem{lemma}{Lemma}
\newtheorem{corollary}{Corollary}
\newtheorem{definition}{Definition}
\newtheorem{conjecture}{Conjecture}
\newtheorem{identity}{Identity}
\newtheorem{axiom}{Axiom}
\newtheorem{problem}{Problem}
\newtheorem{prop}{Proposition}
\newtheorem*{axiom0}{Axiom 0}
\newtheorem*{axiom1}{Axiom 1}

\theoremstyle{remark}
\newtheorem{remark}{Remark}
\newtheorem*{ack}{Acknowledgements}
\newtheorem*{comm}{Comments}
\newtheorem*{aside}{Aside}

\newcommand{\norm}[1]{\left\Vert#1\right\Vert}
\newcommand{\abs}[1]{\left\vert#1\right\vert}
\newcommand{\set}[1]{\left\{#1\right\}}
\newcommand{\sph}{\mathbb{S}}
\newcommand{\df}{\mathrm{d}}
\newcommand{\bmu}{\bar{\mu}}
\renewcommand{\div}{\text{div}}
\newcommand{\divm}{\text{\emph{div}}}

\newcommand{\II}{\operatorname{{I\hspace{-0.2mm}I}}}
\newcommand{\trSigmak}{\tr_\Sigma\!k}

\newcommand{\tr}{{\mbox{tr}}}
\newcommand{\real}{{\bf R}}
\newcommand{\Real}{{\bf R}}
\newcommand{\Ric}{\mbox{Ric}}
\newcommand{\Hess}{\mbox{Hess}}

\newcommand{\R}{\mathbb{R}}
\newcommand{\RP}{\mathbb{RP}}
\newcommand{\Z}{\mathbb{Z}}
\newcommand{\Hcal}{{\mathcal{H}}}
\newcommand{\Ccal}{{\mathcal{C}}}
\newcommand{\Hscr}{{\mathscr{H}}}

\newcommand{\inv}{^{-1}}
\newcommand{\ol}{\overline}
\newcommand{\es}{\emptyset}
\newcommand{\sm}{\setminus}

\title{On curves with nonnegative torsion}
\author{Hubert L. Bray}
\address{Dept. of Mathematics,
Duke University,
Durham, NC 27708}
\email{bray@math.duke.edu}
\thanks{The first named author was supported in part by NSF grant \#DMS-1007063.}

\author{Jeffrey L. Jauregui}
\address{Dept. of Mathematics,
Union College,
Schenectady, NY 12308}
\email{jaureguj@union.edu}

\begin{abstract}
We provide new results and new proofs of results about the torsion of curves in $\R^3$.   Let $\gamma$ be a smooth curve in $\R^3$ that is the graph over a 
simple closed curve in $\R^2$ with positive curvature.  We give a new proof that if $\gamma$ has nonnegative (or nonpositive) torsion, then $\gamma$ has zero 
torsion and hence lies in a plane.  Additionally, we prove the new result that a simple closed plane curve, without any assumption on its curvature, cannot be perturbed to
a closed space curve of constant nonzero torsion.
We also prove similar statements for curves in Lorentzian $\R^{2,1}$ which are related to important open questions 
about time flat surfaces in spacetimes and mass in general relativity. 
\end{abstract}

\maketitle

\section{Introduction}

The curvature and torsion of curves in $\R^3$ are defined by the Frenet formulas found in most undergraduate differential geometry texts 
(see \cite{docarmo}, for instance).  A curve in $\R^3$ with positive constant curvature and nonzero constant torsion must be a helix, 
seen in figure \ref{fig:coil}.  
Positive torsion is what causes the helix to screw in some direction, which in the short run prevents a curve with 
positive torsion from closing up on itself. However, as also shown in figure \ref{fig:coil}, in the long run a curve with positive torsion can circle around and then close up on itself.  
In fact, Weiner even found examples of constant positive torsion curves in $\R^3$ (with non-constant positive curvature) which close up on themselves \cite{weiner} (cf. \cite{bates_melko}).

Nevertheless, there is a clear intuition about torsion to try to exploit here:  positive torsion tends to cause curves to screw in some direction, thereby preventing the curve from being closed.  Hence, it is reasonable to conjecture that a closed curve in $\R^3$ cannot have positive torsion everywhere, as long as an additional hypothesis is added which rules out counterexamples like the one in figure \ref{fig:coil}.  
The additional hypothesis we add, that $\gamma$ is a graph over a simple closed curve in $\R^2$ with positive curvature, is shown in figure \ref{fig:graph}.  
Our first theorem, which is a corollary of results in \cite{sedykh}, \cite{TU},  and \cite{ghomi}, verifies this conjecture:

\begin{theorem}\label{thm:R3}
Let $\gamma$ be a smooth curve in $\R^3$ which is the graph over a simple closed curve in $\R^2$ with positive curvature.  If $\gamma$ has nonnegative (or nonpositive) torsion, then $\gamma$ has zero torsion and hence lies in a plane.  
\end{theorem}
A similar result, which is also a natural adaptation of results in \cite{sedykh}, \cite{TU},  and \cite{ghomi}, is true in Lorentzian $\R^{2,1}$, which has the metric $dx^2 + dy^2 - dt^2$.  In what follows, a tangent vector $v$ is spacelike if the metric
evaluates to be positive on $v$, and a submanifold is spacelike if its nonzero tangent vectors are spacelike.
\begin{theorem}\label{thm:R21}
Let $\gamma$ be a smooth simple closed curve in $\R^{2,1}$ with spacelike curvature vector which lies in a complete spacelike hypersurface.  If $\gamma$ has nonnegative (or nonpositive) torsion, then $\gamma$ has zero torsion and hence lies in a plane.  
\end{theorem}
Note that a complete spacelike hypersurface in $\R^{2,1}$ must be a graph over $\R^2$ (see \cite{harris}, for instance), so that $\gamma$ is again a graph over a simple closed curve in $\R^2$, which turns out to have positive curvature, by the spacelike curvature vector condition (cf. Lemma \ref{lemma_osc}).  Whereas the hypotheses of Theorem \ref{thm:R3} are somewhat \emph{ad hoc}, the hypotheses of Theorem \ref{thm:R21} are quite natural.
We provide new proofs of these theorems in sections \ref{sec:R3} and \ref{sec:R21}.  These new proofs are very efficient and provide a new way of understanding these results.  We also show in section \ref{sec:winding} that these results continue to hold if the projection of $\gamma$ is allowed to wind around the plane curve multiple times.

\begin{figure}
\begin{center}
\includegraphics[height=1in]{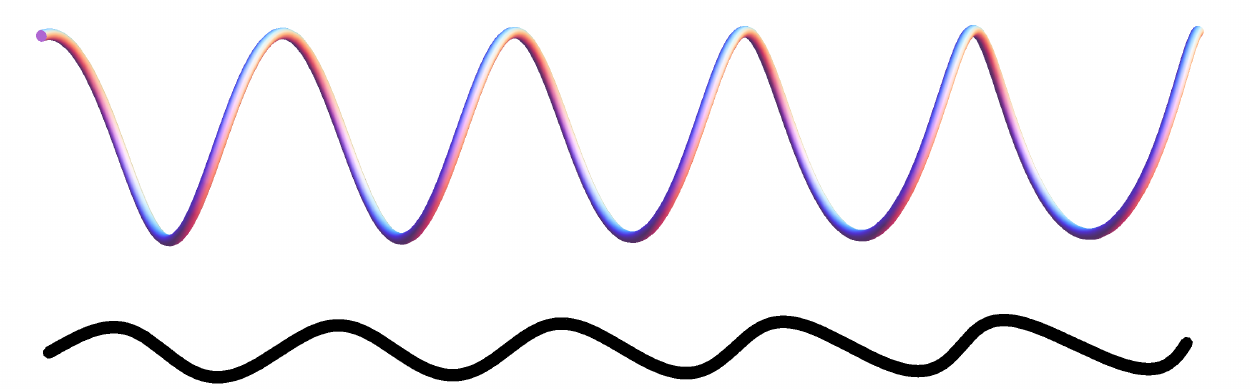}
\hspace{0.1in}
\includegraphics[height=1.5in]{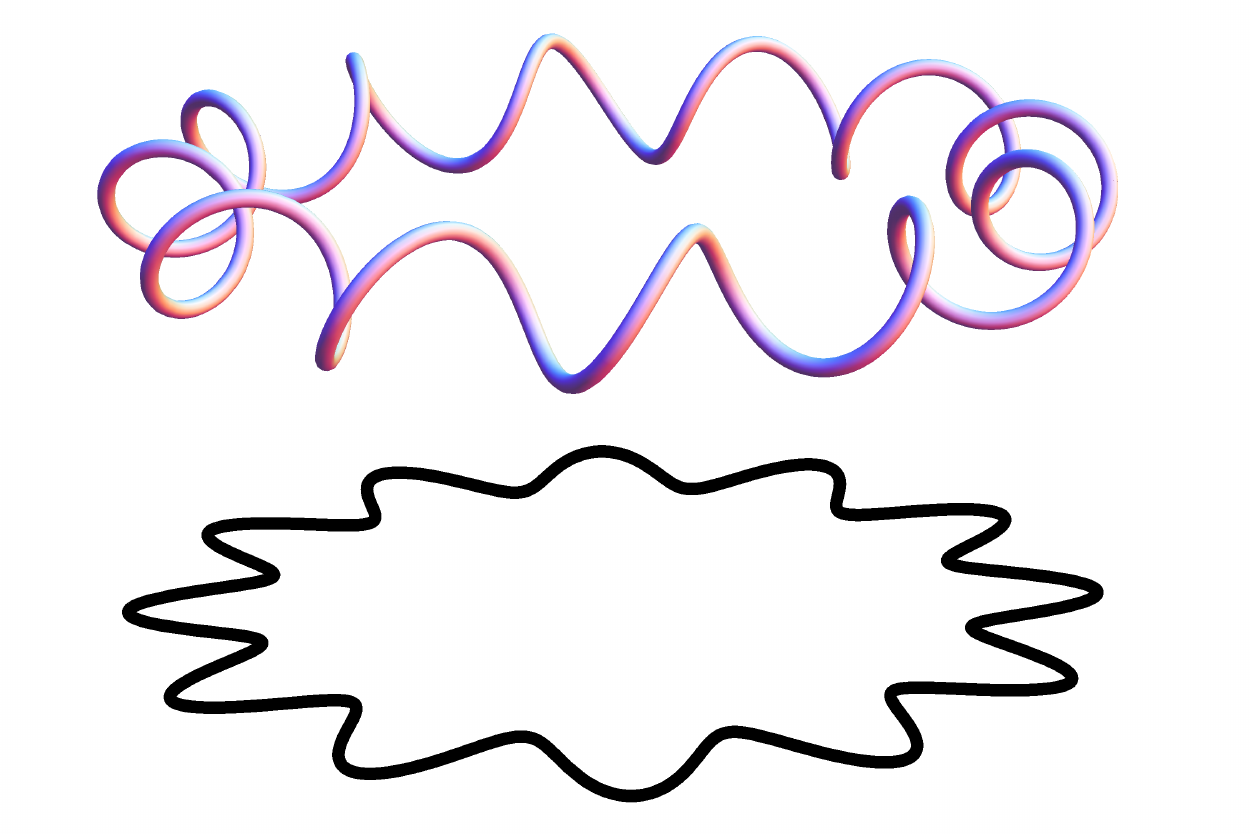}
\caption{\small Two curves in $\R^3$ and their projections to the $xy$ plane.  A curve in $\R^3$ with positive constant curvature and nonzero constant torsion must be a helix, as shown on the left.  This example suggests the conjecture that positive torsion, which is what causes the helix to screw in some direction, should prevent the curve from closing up on itself, at least in the short run.  However, the curve on the right shows how a closed curve in $\R^3$ can have positive torsion everywhere.  Hence, the conjecture will need to exclude this example. \label{fig:coil}}
\end{center}
\end{figure}

In section \ref{sec_rigidity}, we prove a new rigidity result for constant torsion curves (separate from Theorems \ref{thm:R3} and \ref{thm:R21}).  The precise statement is Theorem \ref{thm:rigidity}, but we state an informal version here:

\begin{theorem}
\label{thm:perturb}
A simple closed plane curve, not necessarily convex, cannot be perturbed in the $C^3$ sense to a space curve of constant nonzero torsion.
\end{theorem}

Our interest in this work arose from the study of time flat surfaces \cite{bray_jauregui}, which are spacelike codimension-2 submanifolds of a spacetime satisfying a special geometric condition.  For one-dimensional submanifolds, the condition is simply that of constant torsion.

For the reader's convenience, we recall the Frenet formulas for a curve $\gamma(s)$ in $\R^3$ para\-metrized by arc length, with curvature $\kappa := |\gamma''(s)|$ nonzero.  The tangent 
$T$, normal $N$, and binormal $B$ satisfy:
\begin{align*}
 T' &= \kappa N\\
 N' &= -\kappa T + \tau B\\
 B' &= -\tau N,
\end{align*}
which serves as a definition of the torsion, $\tau$.  Note that conventions for the sign of $\tau$ vary; ours is the opposite of do Carmo \cite{docarmo}.

While we provide new proofs of Theorems \ref{thm:R3} and \ref{thm:R21} in this paper, these two theorems are also implied by previously known four-vertex theorems.  The curves $\gamma$ above are convex in the sense that each point has a tangent plane of which $\gamma$ lies entirely on one side.  The four-vertex theorem of Sedykh \cite{sedykh} for closed convex curves $\gamma$ in $\R^3$ asserts that the torsion of $\gamma$ has at least four zeroes.  
A refinement of this result due to Thorbergsson and Umehara \cite{TU} shows that the torsion must change sign at least four times (if $\gamma$ is not a plane curve), which implies Theorem \ref{thm:R3} (see also \cite{ghomi}).  Theorem \ref{thm:R21} follows as well because the sign of the torsion is determined by the sign of the expression, $(\gamma' \times \gamma'') \cdot \gamma''',$ which 
is independent of whether $\times$ and $\cdot$ are taken with respect to the $\R^3$ or $\R^{2,1}$ metric.

\begin{ack}
The authors would like to thank Daniel Stern and Mark Stern for helpful conversations and are very grateful to Mohammad Ghomi for pointing out the above references and providing insightful feedback.
\end{ack}

\section{Curves in $\R^3$ with nonnegative torsion}
\label{sec:R3}
As depicted in figure \ref{fig:graph}, let $\gamma$ be a smooth curve in $\R^3$ which is the graph, with height function $h$, over a 
simple closed curve $\gamma_0$ in $\R^2$ with curvature $\kappa_0 > 0$.  We begin our discussion by characterizing the geometry of $\gamma_0$.   

\begin{figure}[ht]
\begin{center}
\includegraphics[height=1.9in]{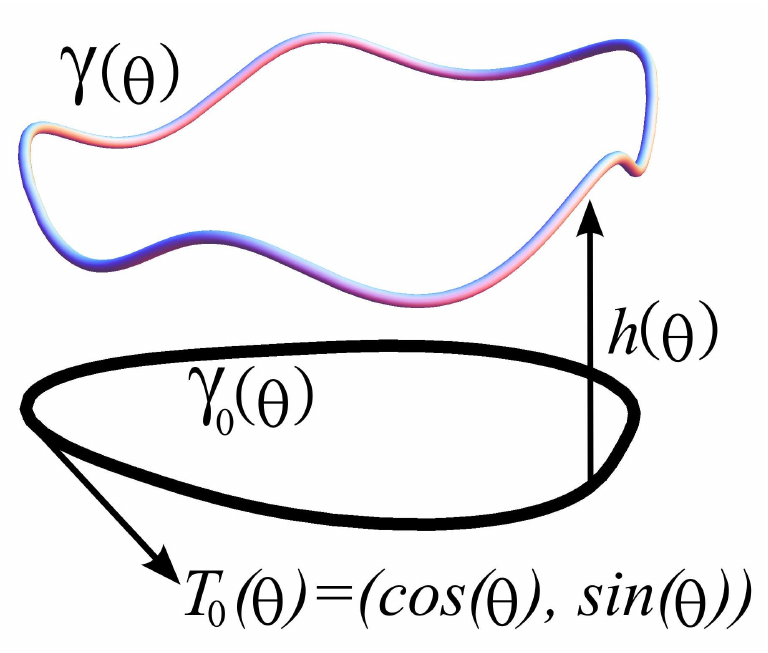}
\caption{\small We restrict to curves $\gamma$ in $\R^3$ that are graphs over simple closed curves $\gamma_0$ in $\R^2$ with positive curvature.  Since $\gamma_0$ has curvature $\kappa_0 > 0$, we may parametrize it by the direction $\theta$ of its unit tangent vector $T_0 = (\cos\theta, \sin\theta)$ in $\R^2$.  Since $\gamma$ is a graph over $\gamma_0$ with height function $h$, we parametrize $\gamma$ and $h$ by $\theta$ as well.\label{fig:graph}}
\end{center}
\end{figure}

Let the unit tangent vector of the base curve $\gamma_0$ in $\R^2$ be 
\begin{equation*}
   T_0 = \gamma_0'(s) = (\cos\theta(s), \sin\theta(s)),
\end{equation*}
where $s$ is the arc length of $\gamma_0$.  A standard exercise is that the (signed) curvature $\kappa_0$ of $\gamma_0$ is given by 
$\kappa_0 = \frac{d\theta}{ds}$ (upon reversing the orientation of $\gamma_0$ if necessary).  Since $\kappa_0 > 0$ and $\gamma_0$ is simple by hypothesis, 
$\theta$ is strictly increasing on $\gamma_0$ and may be taken to go from $0$ to $2\pi$ when 
going around $\gamma_0$ once.  Whereas it is more common to parametrize curves by their arc length parameter $s$, it is essential to our argument to parametrize
$\gamma_0$ (and hence $\gamma$ and $h$ later on) as periodic functions of $\theta$ with period $2\pi$. 

Note
\begin{equation*}
\frac{d \gamma_0}{d\theta} = \frac{1}{\kappa_0}\frac{d\gamma_0}{ds} = \frac{1}{\kappa_0}(\cos\theta, \sin\theta),
\end{equation*}
and, since $\gamma_0$ is closed, $0 = \int_0^{2\pi} \frac{d \gamma_0}{d\theta} d\theta$.  Hence,  
\begin{eqnarray}
   0 &=& \int_0^{2\pi} \frac{1}{\kappa_0(\theta)} \cos\theta \; d\theta \label{eqn:cond1}\\
   0 &=& \int_0^{2\pi} \frac{1}{\kappa_0(\theta)} \sin\theta \; d\theta. \label{eqn:cond2}
\end{eqnarray}
Thus, the geometry of $\gamma_0$ is captured by $\kappa_0(\theta) > 0$, periodic in $\theta$ with period $2\pi$, subject to equations (\ref{eqn:cond1}) and (\ref{eqn:cond2}) above.

\subsection{Proof of Theorem \ref{thm:R3}}

The next step to proving Theorem \ref{thm:R3} is to compute the formula for the torsion of the curve $\gamma$ in $\R^3$ in terms of the height function $h(\theta)$ and the curvature $\kappa_0(\theta)$ of the base curve $\gamma_0$.  For the remainder of the proof, all derivatives are with respect to $\theta$:
\begin{eqnarray*}
\gamma &=& (\gamma_0; h) \\
\gamma' &=& (\frac{1}{\kappa_0} \cos\theta, \frac{1}{\kappa_0} \sin\theta, h') \\
\kappa_0 \gamma' &=& (\cos\theta, \sin\theta, \kappa_0 h') \\
(\kappa_0 \gamma')' &=& (-\sin\theta, \cos\theta, (\kappa_0 h')') \\
(\kappa_0 \gamma')'' &=& (-\cos\theta, -\sin\theta, (\kappa_0 h')'') 
\end{eqnarray*}
so that 
\begin{equation}\label{eqn:cross}
   \kappa_0 \gamma' \times (\kappa_0 \gamma')' = ((\kappa_0 h')' \sin\theta - \kappa_0 h' \cos\theta, - (\kappa_0 h')' \cos\theta - \kappa_0 h' \sin\theta, 1)
\end{equation}
Plugging this into the well known \cite{docarmo} formula (\ref{eqn_tor}) for the torsion $\tau$ of a curve $\gamma$ in $\R^3$ and using standard properties of the cross product gives
\begin{eqnarray}
\tau &=& \frac{(\gamma' \times \gamma'') \cdot \gamma'''}{|\gamma' \times \gamma''|^2} \label{eqn_tor}\\
      &=& \kappa_0 \frac{(\kappa_0 \gamma' \times (\kappa_0 \gamma')') \cdot (\kappa_0 \gamma')''}{|\kappa_0 \gamma' \times (\kappa_0 \gamma')'|^2} \nonumber\\
      &=& \kappa_0 \frac{(\kappa_0 h')'' + \kappa_0 h'}{((\kappa_0 h')')^2 + (\kappa_0 h')^2 + 1} \label{eqn_tor_nice}.
\end{eqnarray}
For the purposes of intuition, consider the simplest case when the base curve $\gamma_0$ is the unit circle in $\R^2$ with constant curvature $\kappa_0 \equiv 1$.  Then 
\begin{equation*}
   \tau = \frac{h''' + h'}{(h'')^2 + (h')^2 + 1}
\end{equation*}
so that 
\begin{equation*}
0 = \int_0^{2\pi} (h'''(\theta) + h'(\theta)) \;d\theta = \int_0^{2\pi} \tau \left[ (h'')^2 + (h')^2 + 1 \right] d\theta
\end{equation*}
from which it is clear that if $\tau \ge 0$ (or $\tau \le 0$), then $\tau \equiv 0$.  Interestingly, this argument generalizes for all 
simple, closed base curves $\gamma_0$ that have $\kappa_0>0$.

\begin{lemma}\label{lem:kernel}
Given a smooth function $\kappa_0(\theta) > 0$ with period $2\pi$ satisfying equations (\ref{eqn:cond1}) and (\ref{eqn:cond2}),  
there exists $f(\theta) > 0$ with period $2\pi$ such that 
\begin{equation}\label{eqn:ode}
   f''(\theta) + f(\theta) = \frac{1}{\kappa_0(\theta)}.
\end{equation}
\end{lemma}
We will prove this lemma momentarily.  Note that in the example above when the base curve $\gamma_0$ is the unit circle in $\R^2$, $f \equiv 1$ satisfies (\ref{eqn:ode}).  This lemma is the key step in the proof of Theorem \ref{thm:R3} since we can then write
\begin{eqnarray}
0 &=& \int_0^{2\pi} \left[f''(\theta) + f(\theta) - \frac{1}{\kappa_0}\right] \kappa_0 h'(\theta) \;d\theta \nonumber\\
&=& \int_0^{2\pi} \left[ (\kappa_0 h')'' + \kappa_0 h' \right] f \;d\theta \nonumber\\
&=& \int_0^{2\pi} \tau \left[ ((\kappa_0 h')')^2 + (\kappa_0 h')^2 + 1 \right] \frac{f}{\kappa_0}  \;d\theta, \label{eqn:avg3}
\end{eqnarray}
having integrated by parts and using (\ref{eqn_tor_nice}).  From this identity,
it is again clear that if $\tau \ge 0$ (or $\tau \le 0$), then $\tau \equiv 0$.  As it is a standard result that a curve with nonzero curvature and zero torsion lies in a plane 
\cite{docarmo}, this proves Theorem \ref{thm:R3}.

\subsection{Proof of Lemma \ref{lem:kernel}}

First note that given $\kappa_0$, a solution $f$ to equation (\ref{eqn:ode}) does not exist unless $\kappa_0$ satisfies the constraints in equations (\ref{eqn:cond1}) and (\ref{eqn:cond2}), as can be seen by integrating by parts twice.  However, given equations (\ref{eqn:cond1}) and (\ref{eqn:cond2}), a {\it positive} solution $f$ to equation (\ref{eqn:ode}) always exists.

To prove existence, we write down the formula for $f$ which comes from identifying the relevant kernel and verify that this $f$ satisfies equation (\ref{eqn:ode}).  For convenience, let $p = 1/\kappa_0$, and recall that $\kappa_0$, $p$, and $f$ are periodic functions with period $2\pi$.  Let 
\begin{equation}\label{eqn:f}
   f(\theta) = \int_{-\pi}^\pi k(\beta) p(\beta + \theta + \pi) \; d\beta
\end{equation}
where
\begin{equation*}
   k(\beta) = \frac{\beta \sin\beta}{2\pi}.
\end{equation*}
Then taking two derivatives in $\theta$ and integrating by parts twice gives us
\begin{equation*}
   f''(\theta) = p(\theta) + \int_{-\pi}^\pi \left(\frac{\cos\beta}{\pi} - \frac{\beta \sin\beta}{2\pi}\right) p(\beta + \theta + \pi) \; d\beta
\end{equation*}
so that 
\begin{eqnarray*}
   f''(\theta) + f(\theta) &=& p(\theta) + \frac{1}{\pi} \int_{-\pi}^\pi p(\beta + \theta + \pi) \cos\beta \; d\beta \\
&=& p(\theta) + \frac{1}{\pi} \int_{\theta}^{\theta + 2\pi} p(\beta) \cos(\beta - \theta - \pi) \; d\beta \\
&=& p(\theta) - \frac{1}{\pi} \int_{\theta}^{\theta + 2\pi} p(\beta) \left(\cos\beta \cos\theta + \sin\beta \sin\theta \right) \; d\beta \\
&=& p(\theta)
\end{eqnarray*}
by equations (\ref{eqn:cond1}) and (\ref{eqn:cond2}).  Since $p = 1/\kappa_0 > 0$ and $k(\beta) > 0$ for $\beta \in (-\pi, \pi) \backslash \{0\}$, it follows from equation (\ref{eqn:f}) that $f > 0$, proving the lemma.

\subsection{Counterexamples to stronger statements}
Theorem \ref{thm:R3} is not true without the assumption of $\kappa_0>0$ for the base curve $\gamma_0$. 
For instance, the closed space curve on the right in figure \ref{fig:coil} has positive torsion and is a graph over a  simple closed plane curve whose
curvature changes sign.  To demonstrate the sharpness of Theorem \ref{thm:R3}, figure \ref{fig:triangle} shows that a simple closed curve with even the slightest amount of negative curvature can admit a graph with positive torsion.  For $t \in [0,2\pi),$ consider the base curve
\begin{equation}
\label{eqn_epi}
\gamma_0(t)= \left((a+b) \cos(t) - c\cos\left(\frac{(a+b)t}{b}\right),(a+b) \sin(t) - c\sin\left(\frac{(a+b)t}{b}\right),0\right),
\end{equation}
with $a=1, b=-1/3,$ and $c=1/6-\epsilon$, which describes a hypotrochoid\footnote{We acknowledge the following website of Mohammad Ghomi, which contains a vast library of plane and space curves: \url{http://people.math.gatech.edu/~ghomi/MathematicaNBs/}}.  For $\epsilon > 0$ small, the curve has $\kappa_0>0$, while for $\epsilon < 0$ small, $\kappa_0$ is slightly negative near the midpoints of  the three sides.  The plane curve in figure \ref{fig:triangle} shows $\gamma_0$ with the choice $\epsilon = -0.05$; the graph on the right uses the height function $h(t) = \frac{1}{4}\sin(3t)$.
By direct computation in Mathematica, $\gamma$ has nonzero curvature and positive torsion (and this behavior persists for all $\epsilon <0$ small).

\begin{figure}
\begin{center}
\includegraphics[height=1.8in]{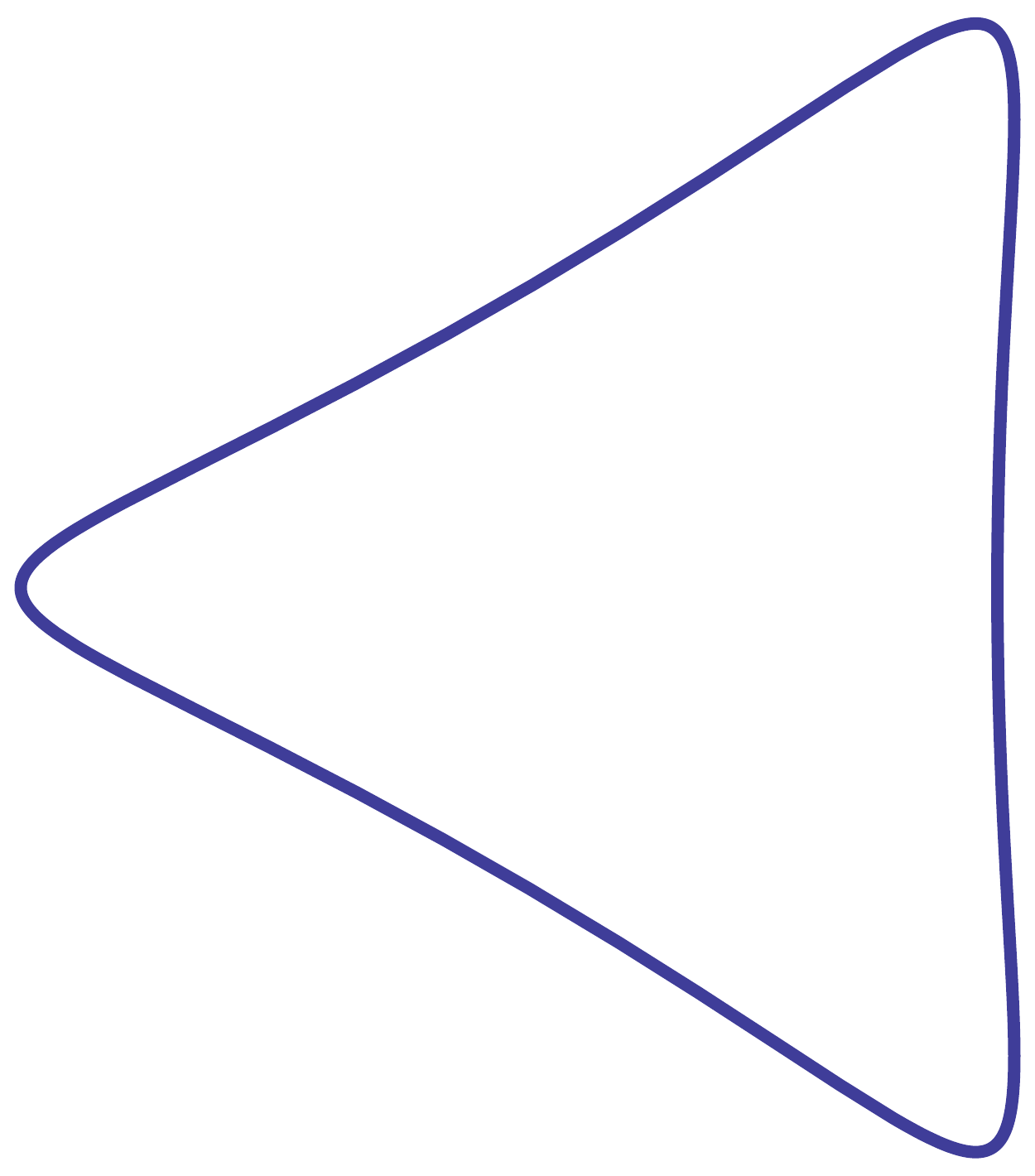}
\hspace{0.5in}
\includegraphics[height=1.25in]{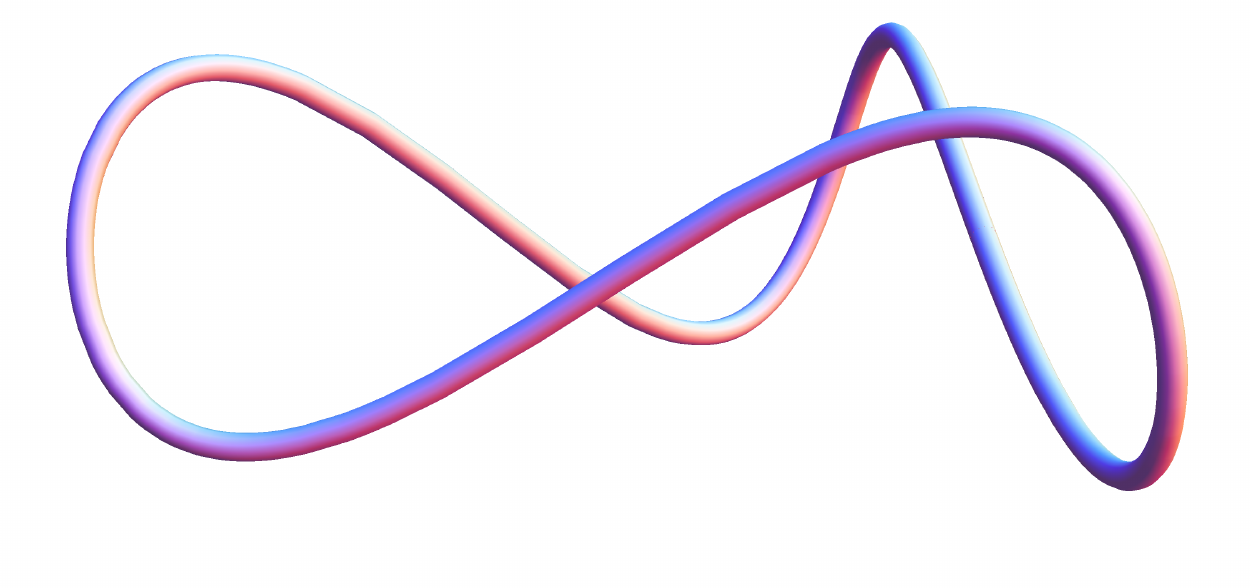}
\caption{\small The rounded triangle on the left, a hypotrochoid, has only a very slight amount of negative curvature, yet the graph over it on the right has positive torsion.  This example demonstrates the sharpness of Theorem \ref{thm:R3}.
\label{fig:triangle}
} \end{center}
\end{figure}

Theorem \ref{thm:R3} is also not true without the assumption that the base curve is simple, that is, does not intersect itself, even if we keep the 
requirement that the curvature of the base curve $\kappa_0 > 0$.  A counterexample can be found using the base curve in (\ref{eqn_epi}), for $t \in [0,2\pi)$, with the values
$a=1, b= 1/5$, and $c=1.3b$.  This curve, an epitrochoid, has positive curvature and is shown on the left in figure \ref{fig:epi} .  The height function $h=\sin(5t)$ yields the curve shown on the right
in figure \ref{fig:epi}.  By direct computation in Mathematica, $\gamma$ has nonzero curvature and positive torsion.

\begin{figure}
\begin{center}
\includegraphics[height=2.1in]{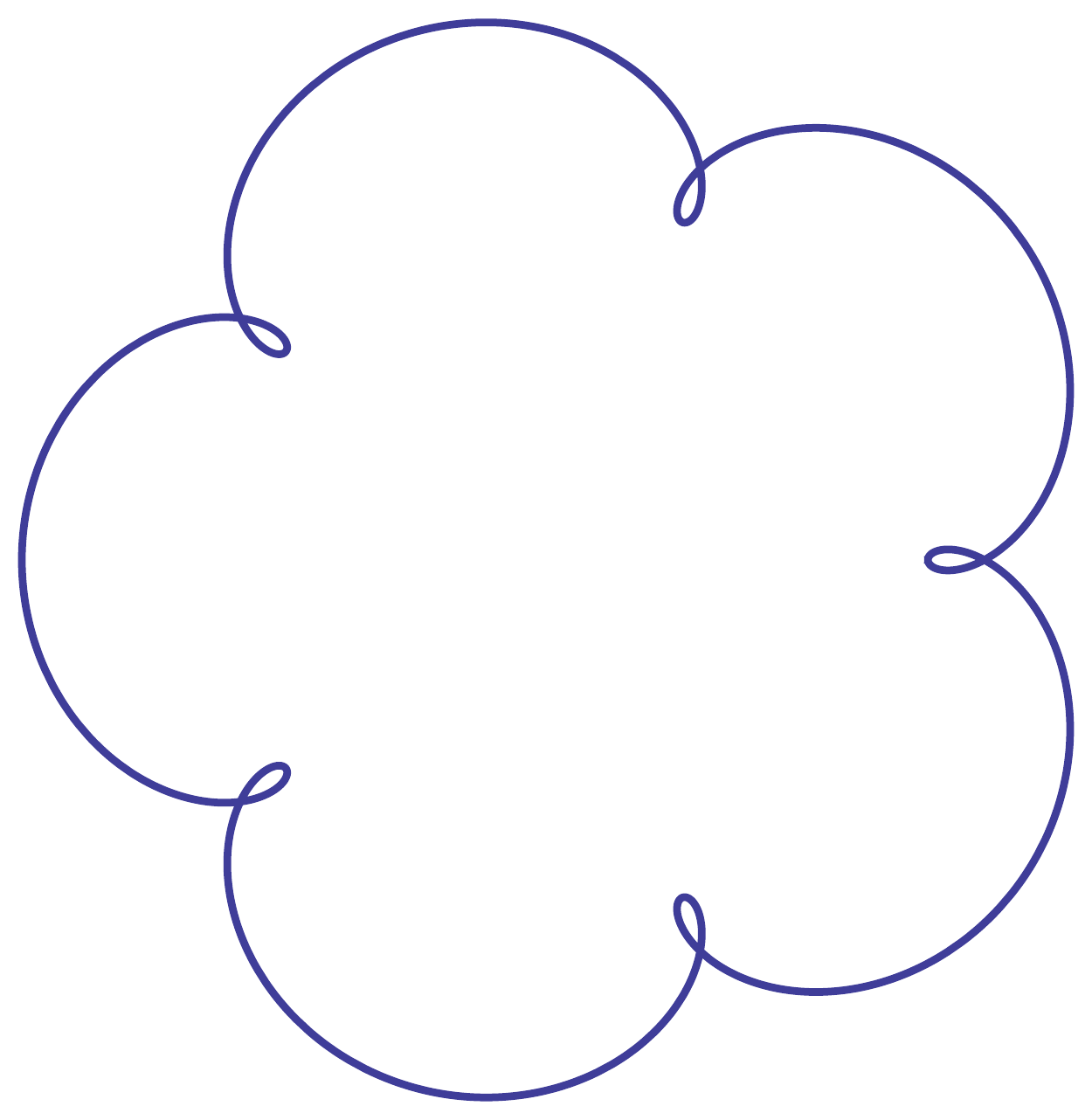}
\includegraphics[height=1.65in]{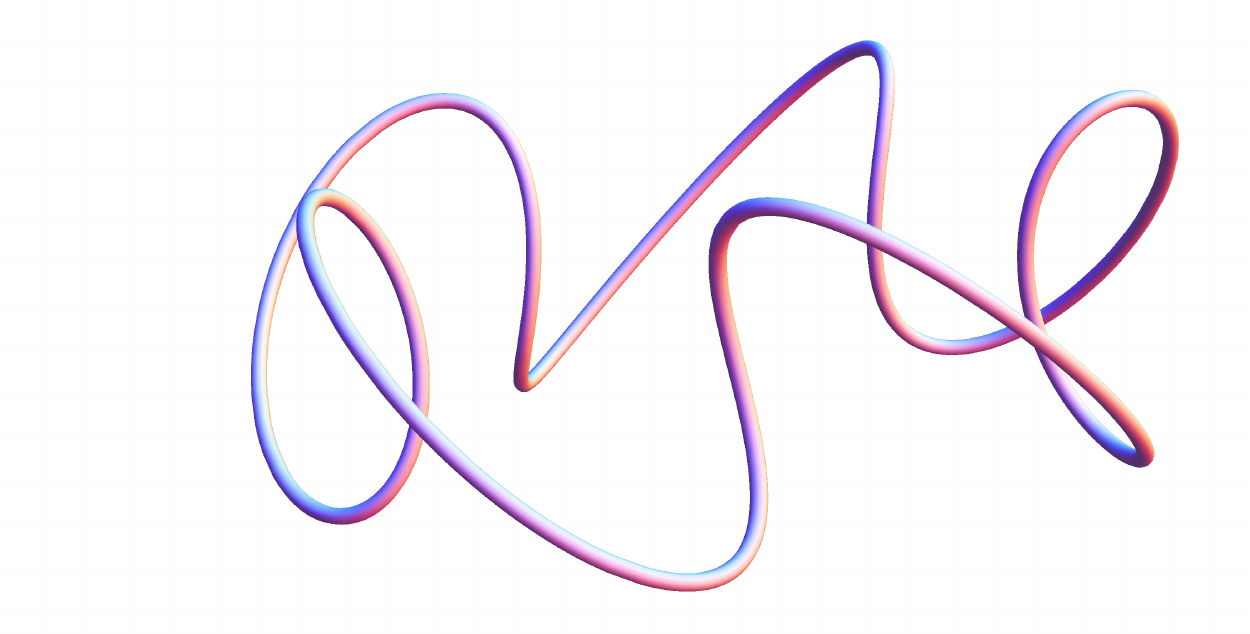}
\caption{\small The closed plane curve $\gamma_0$ on the left, an epitrochoid, is parametrized by (\ref{eqn_epi}) and has positive curvature but is not simple. The space curve $\gamma$ on the right is a graph over $\gamma_0$ with positive torsion.  This example shows that Theorem \ref{thm:R3} is not true without assuming the base curve is simple. 
\label{fig:epi}
} \end{center}
\end{figure}

\section{Curves in $\R^{2,1}$ with nonnegative torsion}
\label{sec:R21}
Since the proof of Theorem \ref{thm:R21} is very similar to the proof of Theorem \ref{thm:R3}, we will only point out the differences.  
A more detailed discussion of the Lorentzian case can be found in \cite{bray_jauregui2}.  

Complete spacelike hypersurfaces in Lorentzian space $\R^{2,1}$ with metric $dx^2 + dy^2 - dt^2$ are well known to be graphs over the $xy$ plane (see \cite{harris} for instance).  Hence, it follows that the smooth simple closed curve $\gamma$ is a graph over a smooth simple closed curve $\gamma_0$ in the $xy$ plane, as in figure \ref{fig:graph}.  

Now consider the osculating plane \cite{docarmo} defined at each point of $\gamma(s)$ by the span of $\gamma'(s)$ and $\gamma''(s)$, where $s$ is the arc length parameter.  The tangent vector $\gamma'(s)$ is spacelike since $\gamma$ lies in a spacelike hypersurface, and $\gamma''(s)$, the curvature vector, is spacelike by assumption (and hence has nonzero length).  Parametrizing by arc length implies that $\gamma''(s)$ is perpendicular to $\gamma'(s)$, so every osculating plane is well-defined and spacelike, and hence not vertical.

\begin{lemma}
\label{lemma_osc}
Consider a curve $\gamma$ in $\R^3$ or $\R^{2,1}$, and suppose the osculating plane is well-defined and is not vertical at some point $\gamma(s_0)$.
Then the projection $\gamma_0$ of $\gamma$ to the $xy$ plane has nonzero curvature at $\gamma_0(s_0)$.
\end{lemma}
\begin{proof}
Inside the osculating plane at $\gamma(s_0)$, there exists a unique osculating circle, which agrees with 
$\gamma$ to second order at $\gamma(s_0)$.  Hence, when we project $\gamma$ and its osculating circle to the $xy$ plane, the resulting ellipse 
will agree with the base curve $\gamma_0$ to second order at every point.  In particular, $\gamma_0$ has nonzero curvature at $s_0$.
\end{proof}
Applying the lemma at every point, and appealing to continuity, the projected curve $\gamma_0$ has curvature $\kappa_0 > 0$.  Hence, we may parametrize everything by 
$\theta$, just as before.

The cross product in $\R^{2,1}$ is defined by simply changing the sign of the $t$ component of the usual cross product in $\R^3$.  Hence, equation (\ref{eqn:cross}) becomes
\begin{equation*}
   \kappa_0 \gamma' \times (\kappa_0 \gamma')' = ((\kappa_0 h')' \sin\theta - \kappa_0 h' \cos\theta, - (\kappa_0 h')' \cos\theta - \kappa_0 h' \sin\theta, -1).
\end{equation*}
(Note that equation (\ref{eqn_tor}) also gives the torsion for curves in $\R^{2,1}$, up to sign.)
Also, the other main difference is that now 
\begin{equation*}
\kappa_0^4 |\gamma' \times \gamma''|^2 = |\kappa_0 \gamma' \times (\kappa_0 \gamma')'|^2 = ((\kappa_0 h')')^2 + (\kappa_0 h')^2 - 1,
\end{equation*}
where $|v|^2 = v \cdot v$ is negative for timelike vectors and positive for spacelike vectors. 

Since $\gamma'$ and $\gamma''$ are linearly independent and spacelike, their cross product $\gamma' \times \gamma''$ must be timelike.  Hence, the above equation implies that we still have a sign on the new term 
\begin{equation*}
((\kappa_0 h')')^2 + (\kappa_0 h')^2 - 1 < 0.
\end{equation*}
Thus, the same argument as before using Lemma \ref{lem:kernel} implies that there exists an $f > 0$ such that 
\begin{equation}\label{eqn:avg21}
0 = \int_0^{2\pi} \tau \left[ ((\kappa_0 h')')^2 + (\kappa_0 h')^2 - 1 \right] \frac{f}{\kappa_0}  \;d\theta 
\end{equation}
from which it is again clear that if $\tau \ge 0$ (or $\tau \le 0$), then $\tau \equiv 0$.  Since the tangent vector and the curvature vector to $\gamma$ are both spacelike, the Frenet frame is well-defined, and so a trivial modification of the usual proof that a curve with zero torsion in $\R^{3}$ lies in a plane \cite{docarmo} works in $\R^{2,1}$ as well.  This proves Theorem \ref{thm:R21}.

\section{Curves of higher winding number}
\label{sec:winding}

Theorems \ref{thm:R3} and \ref{thm:R21} also have nice generalizations when we only assume that the curve $\gamma$  is a local graph over the base curve $\gamma_0$.  
In other words, while $\gamma$ must still project down to $\gamma_0$, it is allowed to wrap around the cylinder over $\gamma_0$ more than once, as in figure \ref{fig:winding}.  The precise statement, which is a corollary to results in \cite{RFS}, goes as follows:

\begin{theorem}
\label{thm:winding}
Let $\gamma_0$ be a simple closed curve in $\R^2$ with positive curvature.
Let $\gamma$ be a smooth closed curve in $\R^3$
that is locally a graph over $\gamma_0$. Then if $\gamma$ has nonnegative (or nonpositive) torsion, then $\gamma$ has zero torsion and hence lies in a plane.
\end{theorem}
The above theorem is also true if $\R^3$ is replaced by $\R^{2,1}$, if we assume $\gamma$ and its curvature vector are spacelike.
Theorem \ref{thm:winding} is depicted in figure 5.  The idea is that positive torsion causes the curve $\gamma$ to ``screw upwards'' in the long run like a helix, thereby making it impossible for it to close up on itself.

Our new proof of Theorem \ref{thm:winding} is nearly exactly the same as our proofs of Theorems \ref{thm:R3} and \ref{thm:R21}.  Our proof is made clear by a modification of equation (\ref{eqn:avg3}), which now becomes:
\begin{equation*}
0 = \int_0^{2\pi k} \tau \left[ ((\kappa_0 h')')^2 + (\kappa_0 h')^2 + 1 \right] \frac{f}{\kappa_0}  \;d\theta.
\end{equation*}
The integer $k$ is the number of times $\gamma$ wraps around the cylinder, where now $h$ and $\tau$ are periodic with period $2\pi k$.  Note that $\kappa_0$ and $f$ still have period $2\pi$ and are defined precisely as before.

An alternate approach is the generalized four-vertex theorem of Romero Fuster and Sedykh \cite{RFS}. Their result shows that on a smooth arc of $\gamma$ joining consecutive self-intersections, at least one zero of the torsion must occur.


\begin{figure}
\begin{center}
\includegraphics[height=1.6in,angle=90]{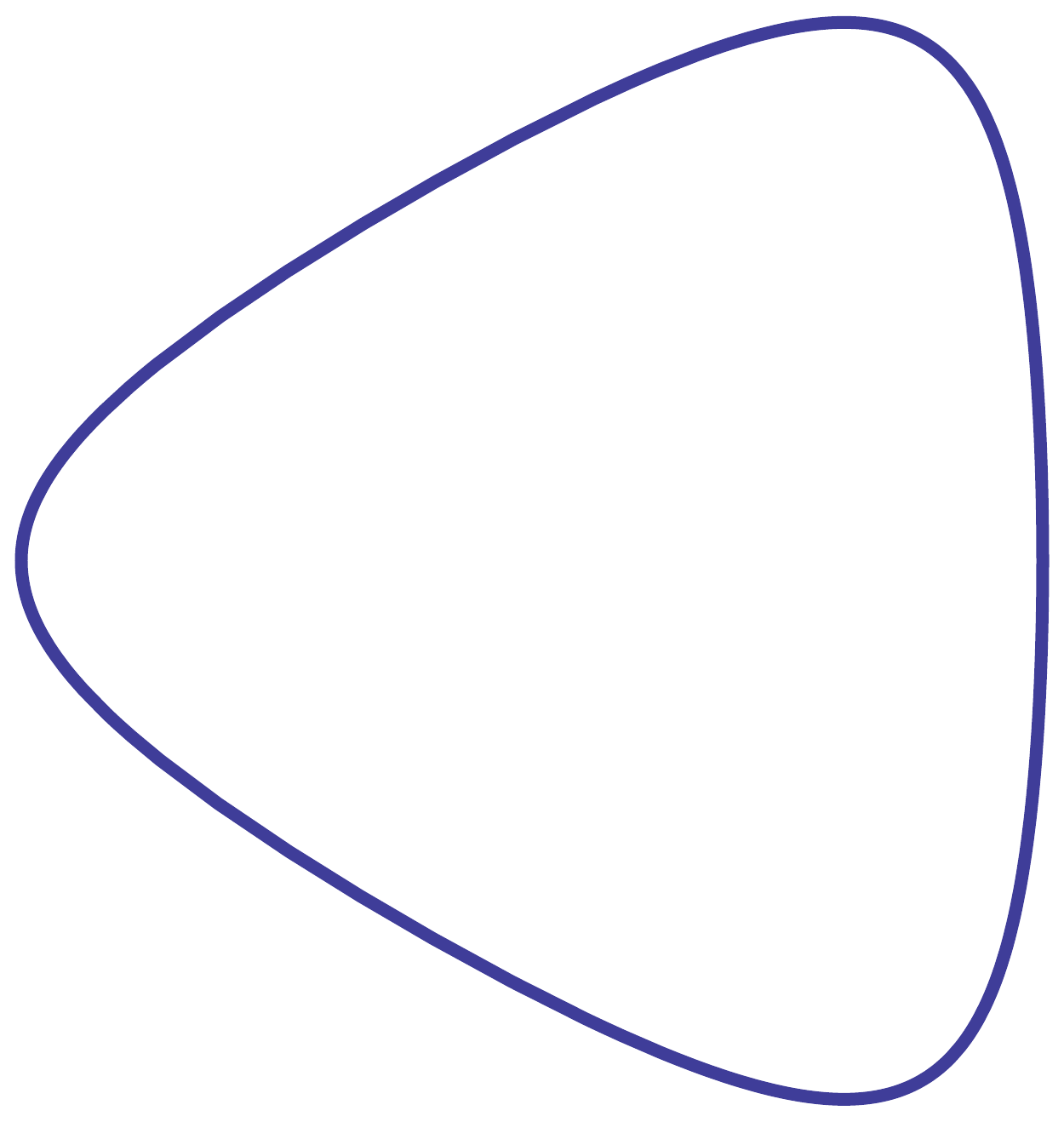}
\includegraphics[height=1.6in]{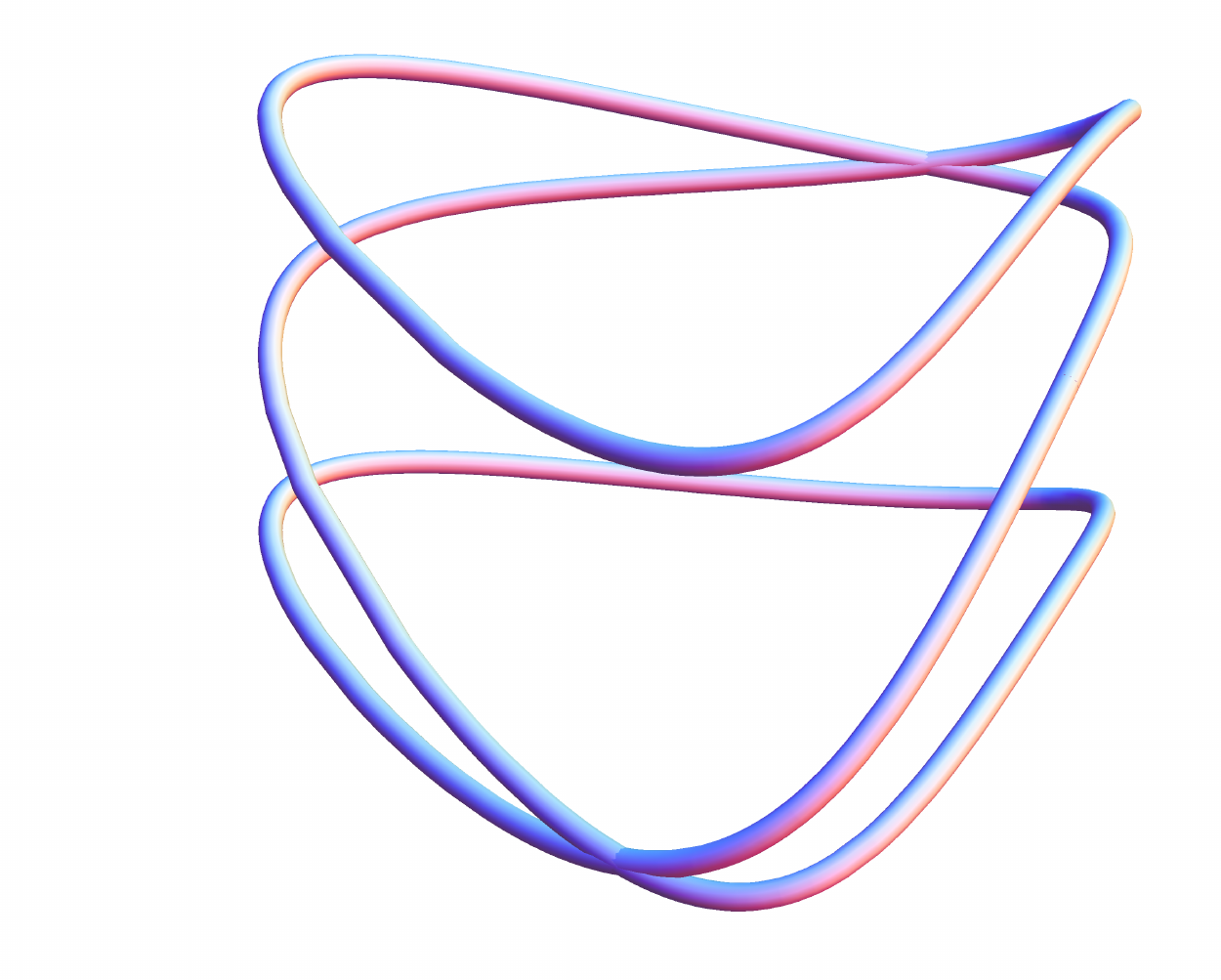}
\includegraphics[height=1.6in]{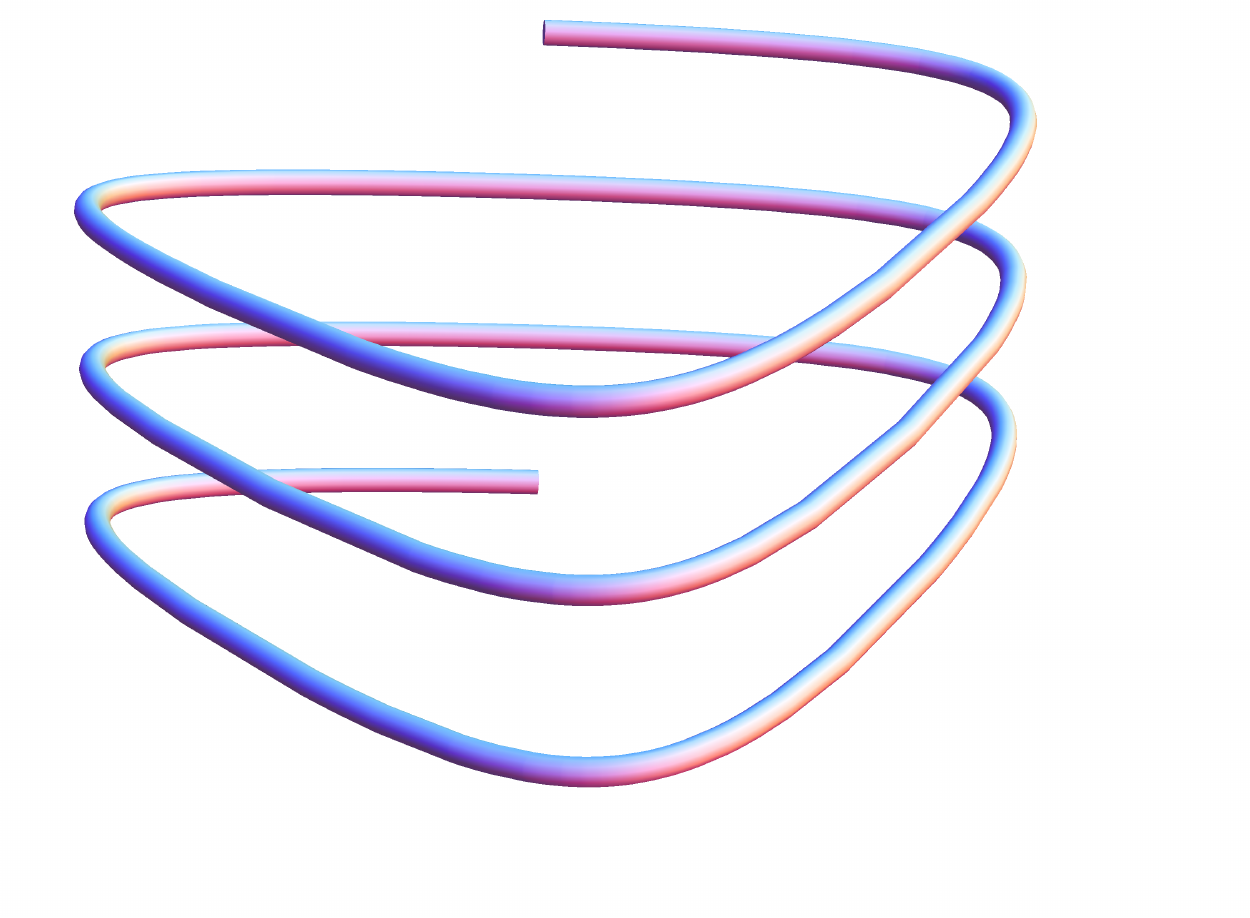}
\caption{\small On the left is a depiction of a plane curve $\gamma_0$ with $\kappa_0>0$.  The space curves in the center and on the right are local graphs over $\gamma_0$, winding around $k=3$ times.  In the center, the torsion of the curve changes signs.  On the right, the torsion is positive, which prevents the curve from being closed (that is, the height function $h$ is not $2\pi k$-periodic).
\label{fig:winding}}
\end{center}
\end{figure}

\section{Rigidity for constant torsion curves}
\label{sec_rigidity}
In this section we pose a rigidity question regarding closed curves of \emph{constant torsion}: is it possible to perturb a plane curve $\gamma_0$ to a space curve with constant nonzero torsion in $\R^3$ (or $\R^{2,1})$?  We consider perturbations in the $C^3$ sense, which is natural because torsion depends on three derivatives of a parametrization.  If $\gamma_0$ has positive curvature, the answer is no, by Theorems \ref{thm:R3} and \ref{thm:R21}, as any small perturbation of such a curve remains a graph over a simple closed plane curve of positive curvature.  In Theorem \ref{thm:rigidity} below, we show that the answer remains ``no'' without any hypothesis on the curvature of $\gamma_0$.

\begin{remark}
It \emph{is} possible to perturb a simple, closed (non-convex) plane curve to have \emph{positive} (but non-constant) torsion.  Consider the coiled helix on the right in figure \ref{fig:coil}.  Its height function may be scaled by a constant $\epsilon>0$, and the torsion of the resulting curve is always positive (and the curvature is positive as well, so the torsion is indeed well-defined).  Yet as $\epsilon \to 0$, the curve converges to its projection to the $xy$ plane, which is simple.
\end{remark}

\begin{remark}
If we drop the hypothesis that $\gamma_0$ is simple (but still require $\kappa_0 > 0$), such constant torsion perturbations exist. Indeed, Weiner's example of a closed curve of constant nonzero torsion belongs to a family of such curves that converges to a self-intersecting plane curve.  

We briefly review this construction (cf. \cite{bates_melko}) here, as it informs the proof of Theorem \ref{thm:rigidity} below. 
For $r \in (0,1]$, define the plane curve
$$\beta_r(t) = r\left(\frac{1}{2}\cos(t) + \frac{\sqrt{2}}{4} \cos(2t), \frac{1}{2}\sin(t) - \frac{\sqrt{2}}{4} \sin(2t),0\right),$$
depicted in figure \ref{fig:beta}.  Let $B_r(t)$ be the vertical lift of $\beta_r(t)$ to the upper hemisphere of the unit sphere.   By an observation of Koenigs  in 1887 \cite{koenigs},
the curve
$$\gamma_r(t) = \frac{1}{r} \int_0^t B_r(t) \times B_r'(t) dt$$
has constant torsion equal to $r$ and binormal equal to $B_r(t)$.  It is possible to verify that $\gamma_r$ is closed with nonvanishing curvature.  
Since the binormal $B_r(t)$ concentrates at the north pole as $r\searrow 0$, it is not hard to see that $\gamma_r$ converges to a plane curve $\gamma_0$ as $r \searrow 0$. It turns out that $\gamma_0$ agrees with $\beta_1$ modulo an isometry of $\R^3$. See figure \ref{fig:nonsimple} for illustrations.
\end{remark}

\begin{figure}
\begin{center}
\includegraphics[height=1.7in]{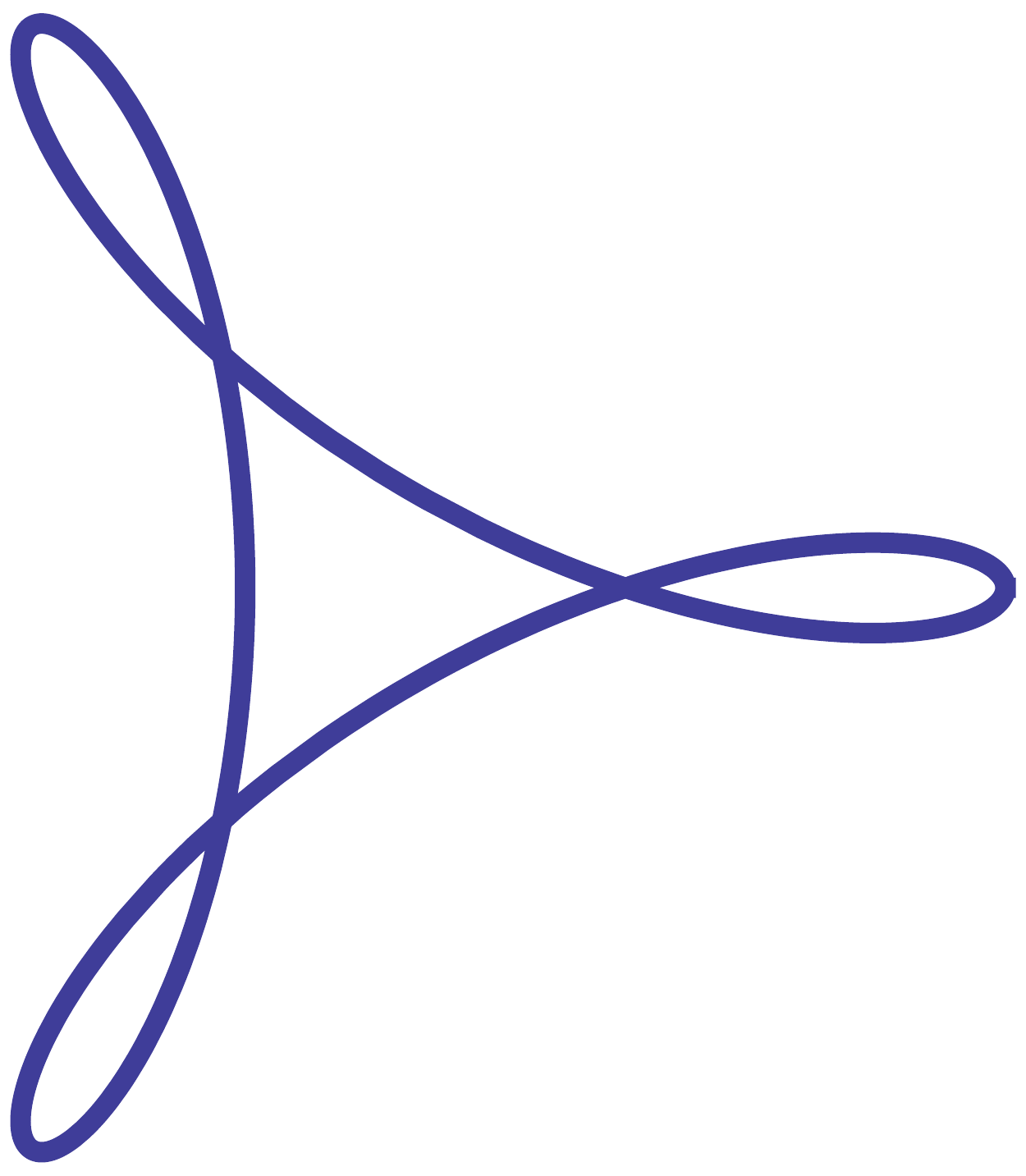}
\caption{\small The plane curve $\beta_1$ used in the construction of a constant torsion curve.  The key features are that $\beta_1$ has a $120^\circ$ rotation symmetry and encloses zero area, counted with multiplicity.  Together, these imply that $\gamma_r$ is closed (see \cites{weiner, bates_melko}).
\label{fig:beta}}
\end{center}
\end{figure}

\begin{figure}
\begin{center}
\includegraphics[height=1.6in]{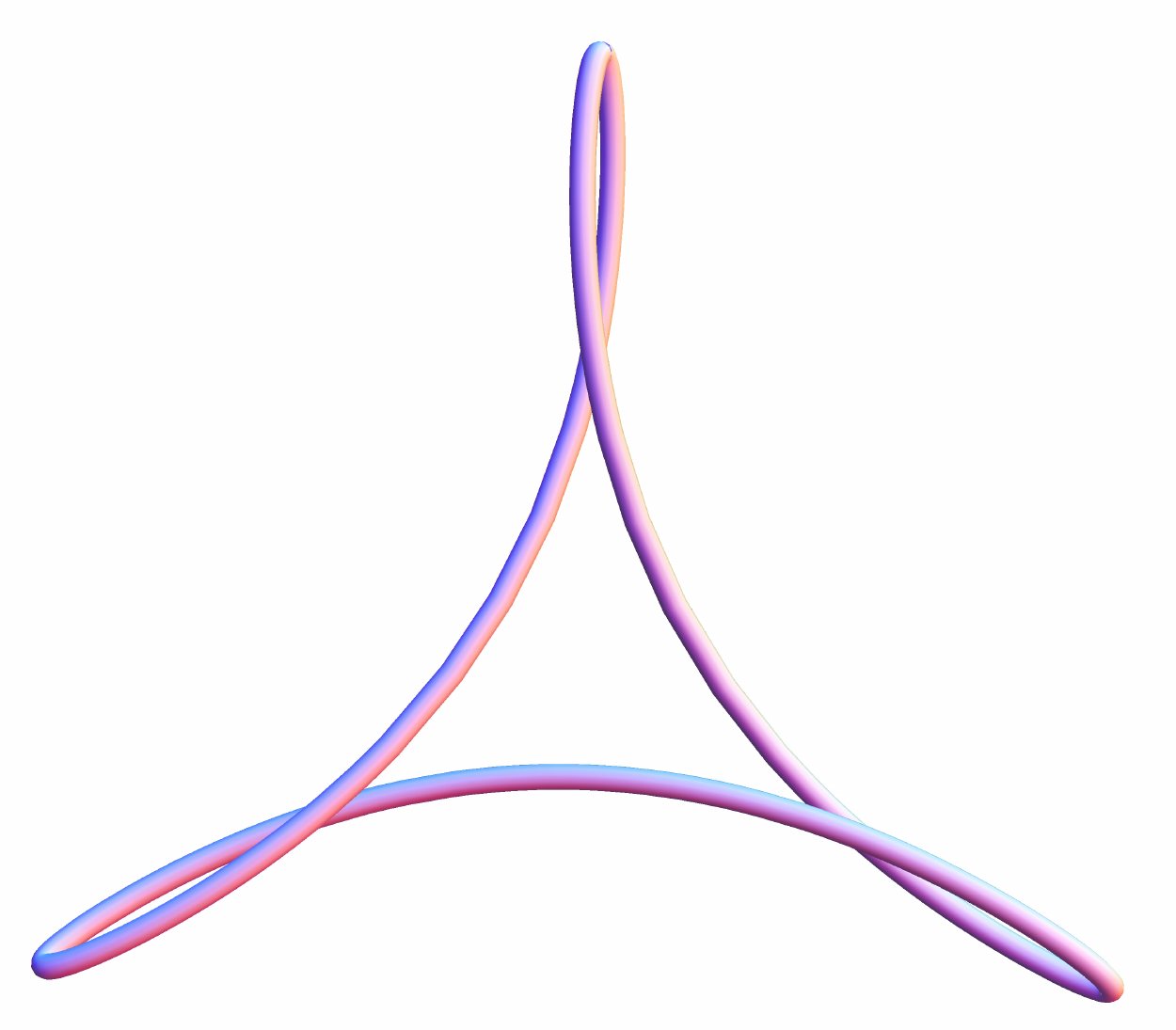}
\includegraphics[width=1.8in]{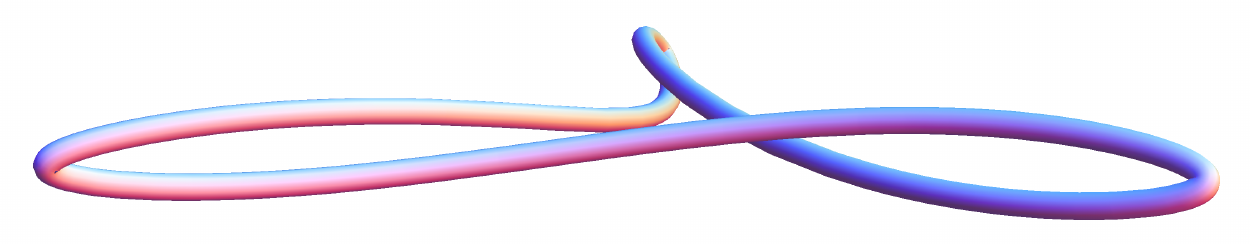}
\includegraphics[height=1.6in]{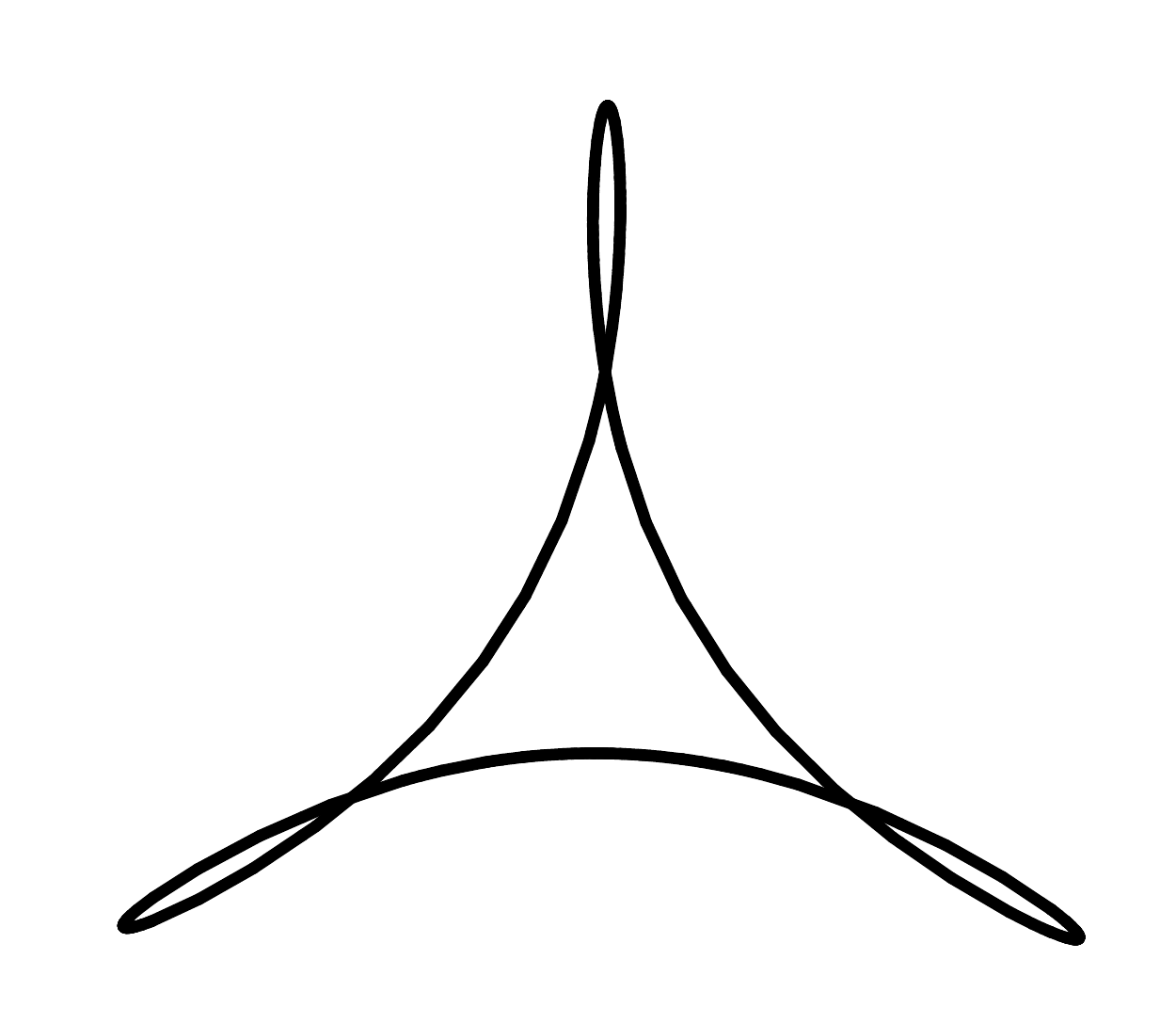}\\
\includegraphics[height=1.6in]{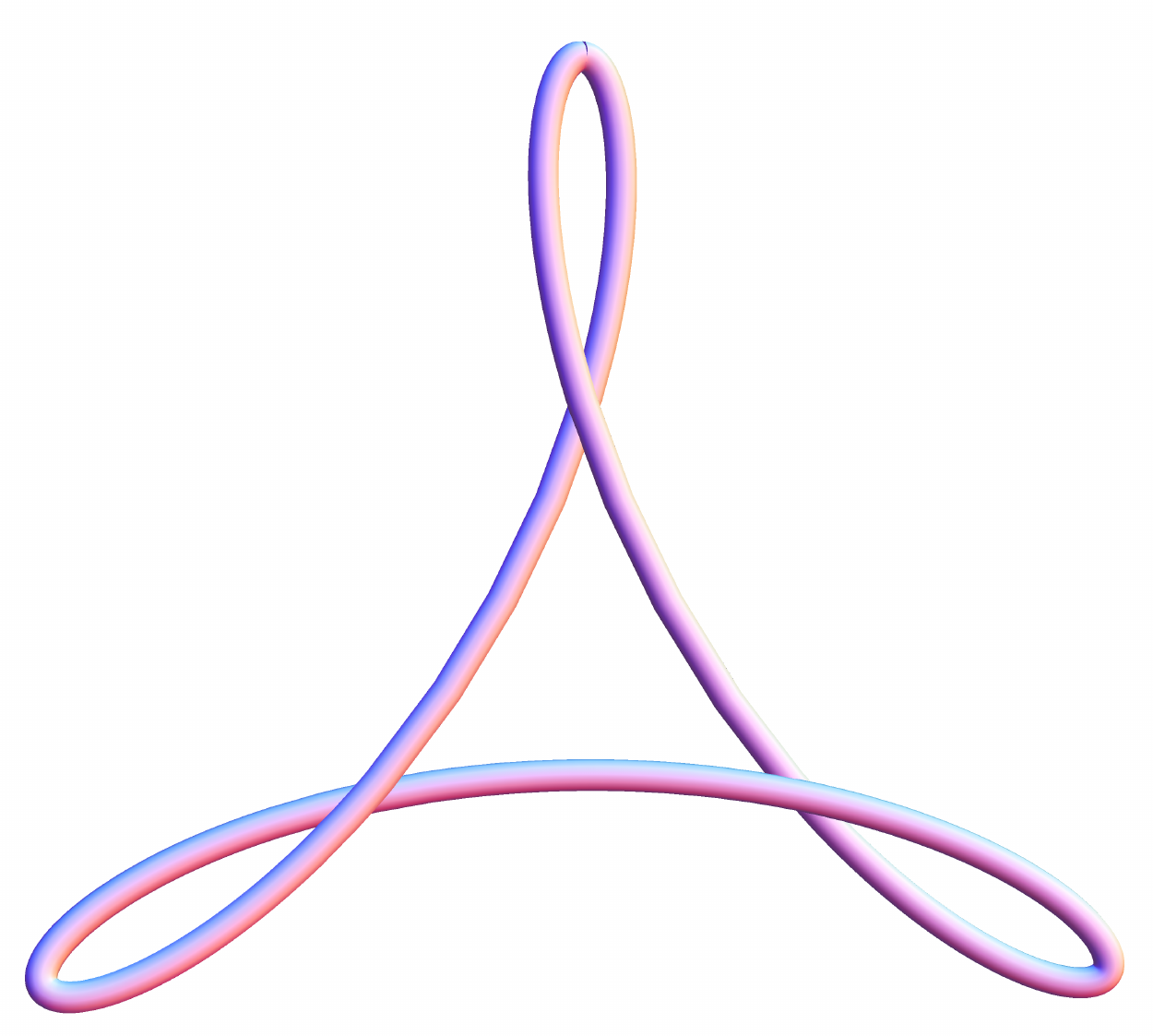}
\includegraphics[width=1.8in]{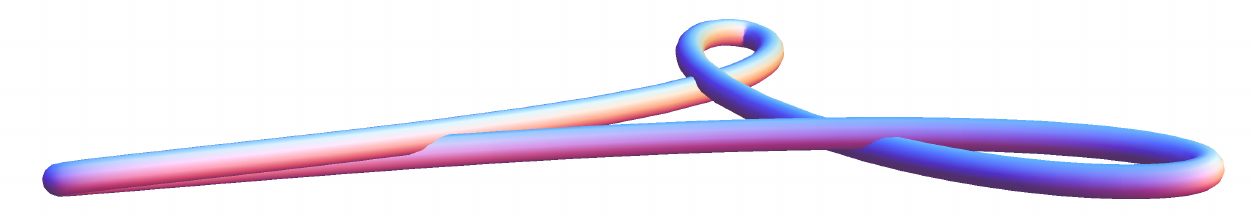}
\includegraphics[height=1.6in]{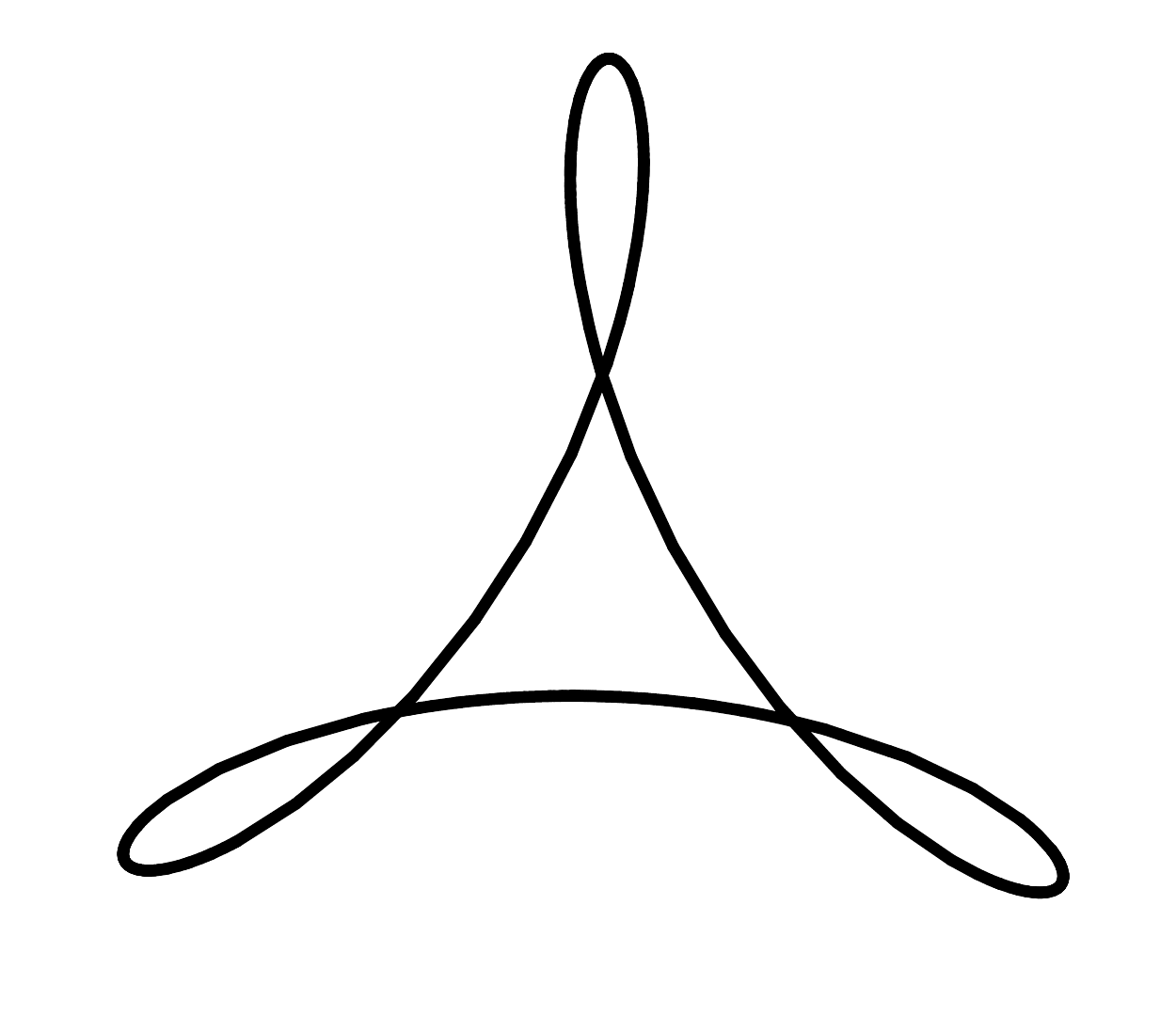}

\caption{\small Depictions of curves of constant torsion constructed by Weiner \cite{weiner} (cf. \cite{bates_melko}).  The top row is $\gamma_1$, shown both from above and from the side, with its projection to the $xy$ plane on the far right.  The bottom row gives the same views for $\gamma_{1/4}$.  As $r \to 0$, $\gamma_r$ converges to the plane curve $\beta_1$, up to a rotation.
\label{fig:nonsimple}}
\end{center}
\end{figure}

Informally, the following theorem states that a simple, closed plane curve cannot be perturbed in the $C^3$ sense to a closed space curve of constant nonzero torsion.

\begin{theorem}
\label{thm:rigidity}
Let $\gamma_n: S^1 \to \R^3$ be a sequence of $C^3$ maps converging in the $C^3$ sense to $\gamma: S^1 \to \R^3$.
 Assume that:
\begin{enumerate}
\item[(i)]  the image of $\gamma$ lies in a plane and is not a point,
\item[(ii)]  $\gamma_n$ has constant speed $c_n > 0$, non-vanishing curvature $\kappa_n$, and constant nonzero torsion $\tau_n \neq 0$.
\end{enumerate}
Then $\gamma$ is not an embedding. In particular, its image has a self-intersection.
\end{theorem}

Theorem \ref{thm:perturb} follows, because a simple closed curve is not embedded.

\begin{proof}[Proof of Theorem \ref{thm:rigidity}]
Let $t$ denote the $S^1$ parameter, and let a prime denote $\frac{d}{dt}$.  By the Frenet formulas, the normal $N_n$ and binormal $B_n$ of $\gamma_n$ satisfy 
\begin{equation}
\label{eqn_B_n_prime}
\frac{1}{c_n}B_n' = -\tau_n N_n.
\end{equation}
Then $B_n \times B_n' = \tau_n \gamma_n'$, since the unit tangent $T_n$ equals $\frac{1}{c_n}\gamma_n'$.  Assuming without loss of generality that $\gamma_n(0)=\vec 0$ for all $n$, we arrive at Koenigs' formula:
$$\gamma_n(t) = \frac{1}{ \tau_n} \int_0^t B_n(t) \times B_n'(t) dt.$$

Without loss of generality, assume that the image of $\gamma$ lies in the $z=0$ plane in $\R^3$.  
\begin{lemma}
\label{lemma_binormal}
The binormal indicatrix $B_n: S^1 \to S^2$ converges
uniformly, as $n \to \infty$, to the constant map with value $(0,0,1)$ (possibly reversing orientation, if necessary). 
\end{lemma}
\begin{proof}
By $C^1$ convergence, $\gamma$ has constant speed $c=\lim_{n \to \infty} c_n \geq 0$; by hypothesis (i), $c> 0$. 
Then there exists a point $p \in S^1$ for which the curvature $\kappa$ of $\gamma$ is not zero, so $\gamma'(p) \times \gamma''(p) \neq \vec 0$.   Then
\begin{align*}
\lim_{n \to \infty} B_n(p) &=\lim_{n \to \infty}  \frac{\gamma_n' \times \gamma_n''}{|\gamma_n' \times \gamma_n''|}(p) =  \frac{\gamma' \times \gamma''}{|\gamma' \times \gamma''|}(p)=(0,0,\pm 1)\\
\lim_{n \to \infty} \tau_n(p) &= \lim_{n \to \infty} \frac{\gamma_n''' \cdot(\gamma_n' \times \gamma_n'')}{|\gamma_n' \times \gamma_n''|^2}(p) =\frac{\gamma''' \cdot(\gamma' \times \gamma'')}{|\gamma' \times \gamma''|^2}(p)  =0,
\end{align*}
by the $C^3$ convergence of $\gamma_n$ to $\gamma$ and the fact that $\gamma$ lies in the $z=0$ plane.  By reversing orientation if necessary, we may assume $B_n(p)$ limits to $(0,0,1)$.  Since $\tau_n$ is constant by hypothesis, we have $\lim_{n \to \infty} \tau_n = 0$.
Now, by (\ref{eqn_B_n_prime}) the speed of the curve $B_n$ is $|B_n'| = c_n |\tau_n|$, which converges to zero.  In particular, $B_n$ converges uniformly to $(0,0,1)$.
\end{proof}

Write $(x_n(t), y_n(t), z_n(t))$ for the components of $B_n(t)$, and let $\beta_n(t) = (x_n(t), y_n(t), 0)$ be its projection to the $z=0$ plane.
Note that $\beta_n$ is a closed plane curve that concentrates at the origin as $n \to \infty$.  In what follows, the limits are in the sense of uniform convergence, and Lemma \ref{lemma_binormal} is used on the third line: 
\begin{align*}
\gamma' &= \lim_{n \to \infty} \gamma_n'\\
&= \lim_{n \to \infty} \frac{1}{ \tau_n} B_n \times B_n'\\
&= \lim_{n \to \infty} \frac{1}{ \tau_n} (0,0,1) \times (x_n', y_n', z_n')\\
&= \lim_{n \to \infty} \frac{1}{ \tau_n} (-y_n', x_n',0)\\
&= A\lim_{n \to \infty} \frac{1}{ \tau_n} ( x_n',y_n',0)\\
&= A\lim_{n \to \infty} \frac{1}{ \tau_n} \beta_n',
\end{align*}
where $A$ is the rotation matrix
$$A = \left[\begin{array}{ccc}
0 & -1 & 0\\
1 & 0 & 0\\
0 & 0 & 1
\end{array}\right].$$
The above demonstrates that the rescaled curves $\frac{1}{ \tau_n} \beta_n$, converge uniformly to $\gamma$ as $n \to \infty$, modulo an isometry of $\R^3$.  Using
Lemma \ref{lemma_signed_area} below, we see that each $\beta_n$ encloses zero area, counted with multiplicity, and thus the same goes for $\frac{1}{ \tau_n} \beta_n$ and
thus for $\gamma$ itself.  This implies that $\gamma$ has self-intersections, which completes the proof of Theorem \ref{thm:rigidity}.
\end{proof}

\begin{lemma} For each $n$,
\label{lemma_signed_area}
$$\int_0^{2\pi} \beta_n \times \beta_n' dt = 0.$$
\end{lemma}
Recall that for a $C^1$ closed curve $\alpha$ in the $z=0$ plane, $(0,0,1) \cdot \int_{S^1} \alpha \times \alpha' dt$ measures the area bounded by $\alpha$, counted with
(possibly positive and negative) multiplicity.
\begin{proof}[Proof of Lemma \ref{lemma_signed_area}]
Since $\gamma_n$ is a closed curve, we have for each $n$:
\begin{align*}
0 &= (0,0,1) \cdot (\gamma_n(2\pi)-\gamma_n(0))\\
 &= (0,0,1) \cdot \int_0^{2\pi} B_n \times B_n' dt\\
&= \int_0^{2\pi} x_n(t)y_n'(t) - y_n(t)x_n'(t) dt\\
&= (0,0,1) \cdot \int_0^{2\pi} \beta_n \times \beta_n' dt.
\end{align*}
This completes the proof, since $\beta_n \times \beta_n'$ is a scalar multiple of $(0,0,1)$ for each $t$.
\end{proof}

\begin{remark}
Theorem \ref{thm:rigidity} generalizes readily for curves in $\R^{2,1}$, assuming $\gamma$ lies in a spacelike plane.  The same proof works with trivial modifications; we remark that the binormal indicatrices $B_n$ take values in the ``unit sphere'' $\{v \in \R^{2,1} : \; |v|^2=-1\}$.
\end{remark}

\section{Discussion}

Below are some obvious corollaries to Theorems \ref{thm:R3} and \ref{thm:R21} with interesting statements:
\begin{corollary}\label{thm:R3-sign}
Let $\gamma$ be a smooth curve in $\R^3$ which is the graph over a simple closed curve in $\R^2$ with positive curvature.  
If $\gamma$ has either positive or negative torsion at a point, then the torsion must have the other sign at some other point.  
\end{corollary}
A similar result is also true in Lorentzian $\R^{2,1}$.
\begin{corollary}\label{thm:R21-sign}
Let $\gamma$ be a smooth simple closed curve in $\R^{2,1}$ with spacelike curvature vector which lies in a complete spacelike hypersurface.  If $\gamma$ has either positive or negative torsion at a point, then the torsion must have the other sign at some other point.
\end{corollary}
Equations (\ref{eqn:avg3}) and (\ref{eqn:avg21}) may be interpreted as saying that the average value of the torsion on the curve $\gamma$ is zero, with respect to a particular choice of positive weighting.  In this manner, one can view these results as a ``weighted total torsion theorem'' (cf. the classical total torsion theorem for curves lying in a sphere \cite{mp}).

Another pair of corollaries comes from considering curves with constant torsion; for the case of Lorentzian $\R^{2,1}$, we describe below connections to general relativity.
\begin{corollary}\label{thm:R3-constant}
Let $\gamma$ be a smooth curve in $\R^3$ which is the graph over a simple closed curve in $\R^2$ with positive curvature.  If $\gamma$ has constant torsion, then $\gamma$ has zero torsion and hence is contained in a plane.  
\end{corollary}
\begin{corollary}\label{thm:R21-constant}
Let $\gamma$ be a smooth simple closed curve in $\R^{2,1}$ with spacelike curvature vector which lies in a complete spacelike hypersurface.  If $\gamma$ has constant torsion, then $\gamma$ has zero torsion and hence is contained in a plane. 
\end{corollary}

In \cite{bray_jauregui} the authors define what it means for a codimension-2 spacelike submanifold of a Lorentzian spacetime to be ``time flat''.  For the case of curves in a (2+1)-dimensional spacetime, time flat is equivalent to constant torsion.  Hence, the previous corollary says that given certain assumptions, time flat curves are contained in planes.  This result supports the choice of terminology: time-flat curves do not bend in timelike directions.
In \cite{bray_jauregui}, we explain how the time flat condition is geometrically natural, along with its importance to understanding the evolution of the 
Hawking mass in general relativity, and describe some interesting conjectures of a purely geometric nature.


\begin{bibdiv}
 \begin{biblist}

\bib{bates_melko}{article}{
   author={Bates, L. M.},
   author={Melko, O. M.},
   title={On curves of constant torsion I},
   journal={J. Geom.},
   volume={104},
   date={2013},
   number={2},
   pages={213--227}
}

\bib{bray_jauregui2}{article}{
   author={Bray, H.},
   author={Jauregui, J.},
   title={On constant torsion curves and time flat surfaces},
   note={Preprint}
}


\bib{bray_jauregui}{article}{
   author={Bray, H.},
   author={Jauregui, J.},
   title={Time flat surfaces and the monotonicity of the spacetime Hawking mass},
   date={2013},
   eprint={http://arxiv.org/abs/1310.8638}
   }

\bib{docarmo}{book}{
   author={do Carmo, M. P.},
   title={Differential geometry of curves and surfaces},
   note={Translated from the Portuguese},
   publisher={Prentice-Hall Inc.},
   place={Englewood Cliffs, N.J.},
   date={1976}
}

\bib{RFS}{article}{
   author={Romero Fuster, M. C.},
   author={Sedykh, V. D.},
   title={On the number of singularities, zero curvature points and vertices
   of a simple convex space curve},
   journal={J. Geom.},
   volume={52},
   date={1995},
   number={1-2},
   pages={168--172}
   }

\bib{ghomi}{article}{
   author={Ghomi, M.},
   title={Tangent lines, inflections, and vertices of closed curves},
   journal={Duke Math. J.},
   volume={162},
   date={2013},
   number={14}
}


\bib{harris}{article}{
   author={Harris, S. G.},
   title={Closed and complete spacelike hypersurfaces in Minkowski space},
   journal={Classical Quantum Gravity},
   volume={5},
   date={1988},
   number={1},
   pages={111--119}
}

\bib{koenigs}{article}{
   author={Koenigs, G.},
   title={Sur la forme des courbes \`a torsion constante},
   journal={Ann. Fac. Sci. Toulouse Sci. Math. Sci. Phys.},
   volume={1},
   date={1887},
   number={2},
   pages={E1--E8}
}

\bib{mp}{book}{
   author={Millman, R. S.},
   author={Parker, G. D.},
   title={Elements of differential geometry},
   publisher={Prentice-Hall Inc.},
   place={Englewood Cliffs, N. J.},
   date={1977}
   }

\bib{sedykh}{article}{
   author={Sedykh, V. D.},
   title={Four vertices of a convex space curve},
   journal={Bull. London Math. Soc.},
   volume={26},
   date={1994},
   number={2},
   pages={177--180}
}



\bib{TU}{article}{
   author={Thorbergsson, G.},
   author={Umehara, M.},
   title={A unified approach to the four vertex theorems. II},
   conference={
      title={Differential and symplectic topology of knots and curves},
   },
   book={
      series={Amer. Math. Soc. Transl. Ser. 2},
      volume={190},
      publisher={Amer. Math. Soc.},
      place={Providence, RI},
   },
   date={1999},
   pages={229--252}
   }

\bib{weiner}{article}{
 author = {Weiner, J.},
 title={Closed curves of constant torsion. II},
 journal={Proc. Amer. Math. Soc.},
 volume={67},
 number={2},
 date={1977},
 pages={306--308}
}
  
 \end{biblist}
\end{bibdiv}
\end{document}